\documentclass[11pt]{article}
\usepackage{amssymb,amsmath,amsthm}
\usepackage{graphicx}
\usepackage{pgf,tikz}

\usepackage{enumerate}

\newtheorem{theorem}{Theorem}
\newtheorem{corollary}[theorem]{Corollary}

\newtheorem{conjecture}[theorem]{Conjecture}

\newtheorem{definition}[theorem]{Definition}

\newtheorem{remark}[theorem]{Remark}

\newcommand{\vanish}[1]{}

\begin{document}
\title{Neighborhood-Prime Labelings of Trees and Other Classes of Graphs}

\author{
Malori Cloys\\
{\small Austin Peay State University}\\
{\small mcloys@my.apsu.edu}\\
\\
N. Bradley Fox\\
{\small Austin Peay State University}\\
{\small foxb@apsu.edu}\\
}

\date{}
\maketitle

\begin{abstract}
A neighborhood-prime labeling of a graph is a variation of a prime labeling in which the vertices are assigned labels from $1$ to $|V(G)|$ such that the $\gcd$ of the labels in the neighborhood of each non-degree $1$ vertex is equal to $1$.  In this paper, we examine neighborhood-prime labelings of several classes of graphs such as polygonal snakes and books, with a focus on trees including caterpillars, spiders, and firecrackers.
\end{abstract}

\section{Introduction}
The concept of a neighborhood-prime labeling of a graph is a variation of a prime labeling, which was developed by Roger Entriger and first introduced in~\cite{Tout} by Tout, Dabboucy, and Howalla.  For a simple graph $G$ with $n$ vertices in the vertex set $V(G)$, a \textit{prime labeling} is an assignment of the integers $1$ to $n$ as labels of the vertices such that each pair of labels from adjacent vertices is relatively prime.  A graph that has such a labeling is called \textit{prime}.  Over the last thirty-five years, prime labelings have been developed or labelings have been shown not to exist on most well-known classes of graphs.  Gallian's dynamic graph labeling survey~\cite{Gallian} contains a detailed list of graphs that have been proven to be prime.

Recently, much focus has been on variations of prime labelings, such as the neighborhood-prime labeling, which involves the neighborhood of a vertex $v$.  This is the set of all vertices in $G$ that are adjacent to $v$ and is denoted by $N(v)$ or $N_G(v)$ when the context of the graph is needed.  A \textit{neighborhood-prime labeling} of a graph $G$ with $n$ vertices is a labeling of the vertex set with the integers $1$ to $n$ in which for each vertex $v\in V(G)$ of degree greater than $1$, the $\gcd$ of the labels of the vertices in $N(v)$ is 1.  This labeling is often represented by a bijective function $f:V(G)\rightarrow\{1,2,\ldots, n\}$ with $\gcd \{f(u): u\in N(v)\}=1$ for all non-degree $1$ vertices.  For ease of notation, we will use $f(N(v))$ to denote $\{f(u): u\in N(v)\}$.  A graph which admits a neighborhood-prime labeling is called \textit{neighborhood-prime}.

Patel and Shrimali developed the notion of a neighborhood-prime labeling, first introduced in~\cite{Patel_Shrimali1}, where they proved all paths, complete graphs, wheels, helms, and flowers are neighborhood-prime.  Additionally, a cycle $C_n$ was shown to be neighborhood-prime for all $n\not\equiv 2 \pmod{4}$.  Patel and Shrimali~\cite{Patel_Shrimali2, Patel_Shrimali3} have investigated other classes of graphs in subsequent papers, including grid graphs and unions of paths and wheels being neighborhood-prime, along with the union of cycles given certain restrictions on their lengths.  Other classes of graphs shown to be neighborhood-prime include the friendship graph~\cite{A_N}, bistar~\cite{A_M}, and crown graph~\cite{K_V}.

\section{Graphs with Cycles}

We begin our examination of neighborhood-prime labelings with the familiar classes of gear and snake graphs. Ananthavalli and Nagarajan claimed in~\cite{A_N} that gear graphs and triangular and quadrilateral snakes are neighborhood-prime.  However, their labelings are not neighborhood-prime in many cases; see for instance the neighborhoods of the vertices with odd indices in Theorem 2.2 for any size $n$, or the neighborhood of $v_{2n+1}$ in Theorem 2.4 for any length $n$.  We first provide correct labelings for these graphs, and then expand upon the study of polygonal snakes.  Our first graph, the \textit{gear graph} $G_n$, is obtained from a wheel graph $W_n = C_n + K_1$ by adding a vertex between each pair of adjacent vertices on the cycle.  See Figure~\ref{gear} to see a neighborhood-prime labeling for the gear graph with $n=7$.

\begin{theorem}\label{gear_graph}
The gear graph $G_n$ is neighborhood-prime for every $n$.
\end{theorem}

\begin{proof}
	Consider the graph $G_n$ with $v_1$ as the center vertex and $v_2,\ldots,v_{2n+1}$ as the cycle vertices in which $v_{2i+1}$ is adjacent to the center for each $i=1,\ldots,n$.
\smallskip

First, we assume $n=3k$ or $n=3k+2$ for some positive integer $k$.  We define a labeling for the graph simply by $f(v_i)=i$ for all $i$.  Let $v\in V(G_n)$.  We need to show $\gcd\{f(N(v))\}=1$, which we do in four distinct cases.
\medskip

\noindent\textbf{Case 1}: If $v=v_1$, then we see that  $N(v_1)=\{v_3,v_5,\ldots,v_{2n+1}\}$. Therefore, $\gcd\{f(N(v_1))\}=\gcd\{3,5,\ldots,2n+1\}=1$ since the neighborhood of the vertex $v_1$ contains consecutive odd integers.
	
\medskip
	
\noindent\textbf{Case 2}: For $v=v_{2i}$ with $i=2,\ldots,n$, we see that $N(v_i)=\{v_{2i-1},v_{2i+1}\}$. Clearly we have $\gcd\{2i-1,2i+1\}=1$ for vertices in this case since each neighborhood consists of only vertices labeled by consecutive odd integers.
	
\medskip
\noindent\textbf{Case 3}: If $v=v_{2i+1}$ for $i=1,\ldots, n$, then we see that $v_1\in N(v_{2i+1})$. Therefore, for the odd-indexed vertices, we have $\gcd\{f(N(v_{2i+1}))\}=1$.

\medskip
\noindent\textbf{Case 4}: Lastly, if $v=v_2$, then we see that $N(v_2)=\{v_3, v_{2n+1}\}$.   When $n=3k$, we have $2n+1=6k+1$, and when $n=3k+2$, we obtain $2n+1 = 6k+5$.  One can see that in either case, $2n+1$ is not a multiple of 3, thus we have $\gcd\{3, 2n+1\}=1$.

\bigskip 
Now we consider when $n=3k+1$.  We assign labels as before with $f(v_i)=i$ except for the final two odd-indexed vertices, in which we set $f(v_{2n-1})=2n+1$ and $f(v_{2n+1})=2n-1$.  Let $v\in V(G_n)$.  The first three cases follow analogously to when $n=3k$ or $3k+2$ if we limit Case 2 to only $i=2$ to $n-2$, but this time we need three other cases to complete the proof.

\medskip 
\noindent\textbf{Case 4}: When $v=v_2$, we see that $N(v_2)=\{v_3, v_{2n+1}\}$ as before and therefore $\gcd\{f(N(v_2))\}=\gcd\{3,2n-1\}=1$ since $3$ does not divide $2n-1=6k+1$ in this case. 

\medskip

\noindent\textbf{Case 5}: For $v=v_{2n}$, we have $N(v_{2n})=\{v_{2n-1},v_{2n+1}\}$. Thus $\gcd\{2i+1,2i-1\}=1$ since these labels are consecutive odd integers.

\medskip
\noindent\textbf{Case 6}: If $v=v_{2n-2}$, then $N(v_{2n-2})=\{v_{2n-3},v_{2n-1}\}$.  We see that $\gcd\{2n-3,2n+1\}=1$ because all odd integers that differ by $4$ are relatively prime.

\medskip 
We have shown for all vertices $v\in G_n$ in each case of $n=3k$, $3k+1,$ and $3k+2$, that $\gcd\{f(N(v))\}=1$, making the gear graph neighborhood-prime.

\end{proof}

\begin{figure}
\begin{center}
\includegraphics[scale=1]{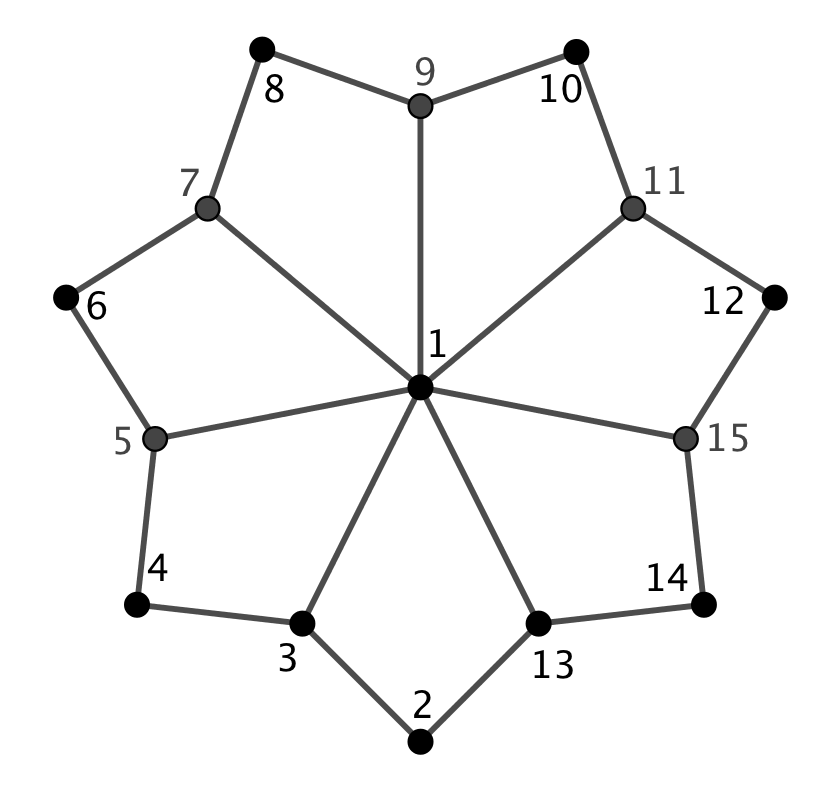}
\caption{Neighborhood-prime labeling of the gear graph $G_7$}\label{gear}
\end{center}
\end{figure}

We now investigate snake graphs, which we will denoted by $S_{k,n}$, where $k$ is the length of the cycles attached along a path of $n$ vertices.  We begin with the \textit{triangular snake}, $S_{3,n}$, which is denoted in other works as $T_n$.  This graph consists of a path $P_n$ with a vertex attached to the pairs of adjacent vertices from the path, resulting in $n-1$ triangles whose bases form a path.  An example of a triangular snake with a neighborhood-prime labeling can be observed in Figure~\ref{triangular}.

\begin{theorem}\label{triangular_snake}
The triangular snake admits a neighborhood-prime labeling for any length $n$.
\end{theorem}

\begin{proof}
We call the $n$ vertices on the base path $u_1,\ldots,u_n$ and the vertex attached to form triangles between each pair of vertices on the path $v_i$ for each $i=1,\ldots,n-1$. We assign $f(u_i)=2i-1$ for $i=1,\ldots,n$ and $f(v_i)=2i$ for $i=1,\ldots,n-1$. 

Let $v\in V(T_n)$. If $v=v_i$ for any $i$, then $N(v_i)=\{u_i, u_{i+1}\}$, which are labeled with consecutive odd integers. If $v=u_i$ for $i=1,\ldots,n-1$, then $f(v_i)=2i$ and $f(u_{i+1})=2i+1$, resulting in consecutive integer labels in its neighborhood. Lastly, when $v=u_n$, the labels in its neighborhood are again consecutive integers with $f(v_{n-1})=2n-2\in f(N(u_n))$ and $f(u_{n-1})=2n-3\in f(N(u_n))$. In all three cases, $gcd\{f(N(v))\}=1$.

\end{proof}

\begin{figure}
\begin{center}
\includegraphics[scale=1]{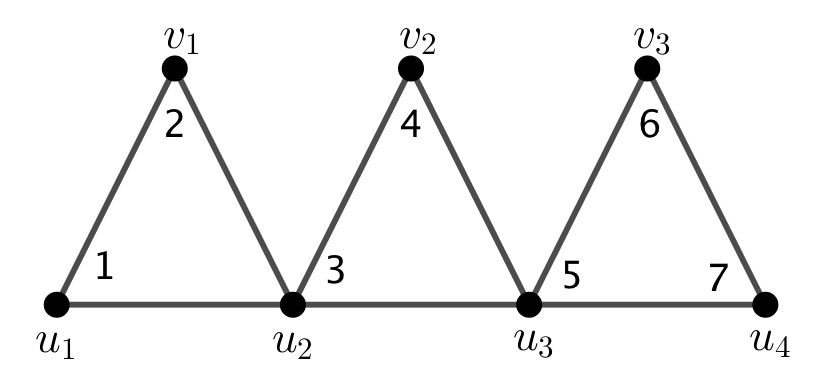}
\includegraphics[scale=1]{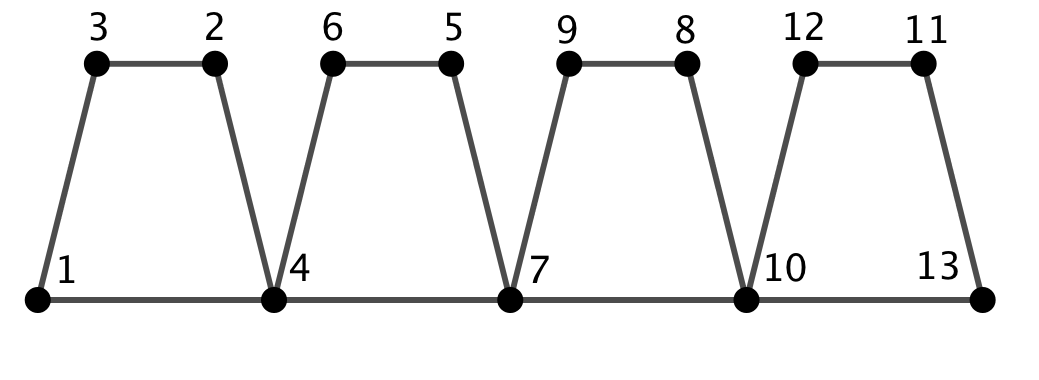}
\caption{Neighborhood-prime labeling of the triangular snake $S_{3,4}$ and quadrilateral snake $S_{4,5}$}\label{triangular}
\end{center}
\end{figure}

The \textit{quadrilateral snake} $S_{4,n}$, or denoted in~\cite{A_N} by $Q_n$, is the graph resulting from a path $u_1,\ldots,u_n$ by attaching a path with four vertices $u_i,v_i,w_i,u_{i+1}$ between each pair of adjacent vertices on the path.  See Figure~\ref{triangular} for an example of $S_{4,5}$.

\begin{theorem}\label{quad_snake}
	The quadrilateral snake is neighborhood-prime for all $n$.
\end{theorem}

\begin{proof}
	We label the quadrilateral snake $S_{4,n}$ as follows:
\begin{align*}
f(u_i)&=3i-2 \text{\hspace{.2cm} for } i=1,\ldots, n\\
f(v_i)&=3i \text{ \hspace{.875cm}for } i=1,\ldots, n-1\\
f(w_i)&=3i-1 \text{\hspace{.2cm} for } i=1,\ldots, n-1.
\end{align*}
One can easily verify that this labeling is neighborhood-prime since the neighborhood of each vertex $v\in V(S_{4,n})$ contains vertices with consecutive integer labels.
\end{proof}

While the triangular and quadrilateral snake have been previously studied, although with incorrectly developed labelings, the \textit{pentagonal snake} $S_{5,n}$ has not been previously investigated.  It is similarly obtained from a path $u_1,\ldots,u_n$ by attaching path between the vertices $u_i,v_i,w_i,x_i,u_{i+1}$ between each pair of adjacent vertices on the path $P_n$.  An example can be seen in Figure~\ref{starmngon}.

\begin{theorem}\label{pentagon_snake}
	The pentagonal snake $S_{5,n}$ has a neighborhood-prime labeling for any length $n$.
\end{theorem}

\begin{proof}
We begin by labeling the pentagonal snake as follows:
	\begin{align*}
	f(u_i)&=4i-3 \text{\hspace{.2cm} for } i=1,\ldots,n\\
	f(v_i)&=4i-1 \text{\hspace{.2cm} for } i=1,\ldots,n-1\\
	f(w_i)&=4i \text{\hspace{.875cm} for } i=1,\ldots,n-1\\
	f(x_i)&=4i-2 \text{\hspace{.2cm} for } i=1,\ldots,n-1.
	\end{align*}
\medskip

To avoid having multiples of $3$ as the labels of the only two neighboring vertices of each $v_i$ for $i=3k$, we reassign the following labels:
\begin{align*}
	f(u_i)&=4i-1 \text{\hspace{.675cm} for } i=3k \text{ with }i<n\\
	f(v_i)&=4i-3 \text{\hspace{.675cm} for } i=3k. 
	\end{align*}

Let $v\in V(S_{5,n})$.  We consider five cases to demonstrate the equality $\gcd\{f(N(v))\}=1$.

\noindent\textbf{Case 1:} Assume $v=u_1$.  The labels in the neighborhood of $u_1$ are $3$ and $5$, which are relatively prime.
\medskip

\noindent\textbf{Case 2:} Suppose $v=u_i$ for $i=2,\ldots, n$.  The neighborhood of each $u_i$ contains the vertices $u_{i-1}$ and $x_{i-1}$, which are labeled by consecutive integers.  Note that this is true for $u_{3k+1}$ as well despite the reassigned label of $f(u_{3k})$.  Thus, $\gcd\{f(N(u_i))\}=1$. 
\medskip

\noindent\textbf{Case 3:} Suppose $v_i$ is one of the following vertices: $v_i$ with $i=3k$, $w_i$ for any $i=1,\ldots, n-1$, or $x_i$ with $i\neq 3k+2$.  Each of these vertices has exactly two neighbors labeled by consecutive integers, satisfying the coprime neighborhood condition.
\medskip

\noindent\textbf{Case 4:} Next assume $v=v_i$ with $i\neq 3k$.  We have $N(v_i)=\{u_i, w_i\}$, which results in $\gcd\{f(N(v_i))\}=\gcd\{4i-3, 4i\}=1$ since we assumed $i$ is not a multiple of $3$.
\medskip

\noindent\textbf{Case 5:} Finally suppose that $v=x_i$ with $i=3k+2$.  These vertices have a neighborhood of $\{w_i, u_{i+1}\}$.  Since $i+1$ is a multiple of three, the reassigned label would give us $\gcd\{f(N(x_i))\}=\gcd\{4i, 4(i+1)-1\}=\gcd\{4i, 4i+3\}=1$ because of the assumption of $i$ not being a multiple of $3$.
\medskip

We see that $\gcd\{f(N(v))\}=1$ for every vertex $v$, proving that this labeling is neighborhood-prime.
\end{proof}

\begin{figure}
\begin{center}
\includegraphics[scale=.75]{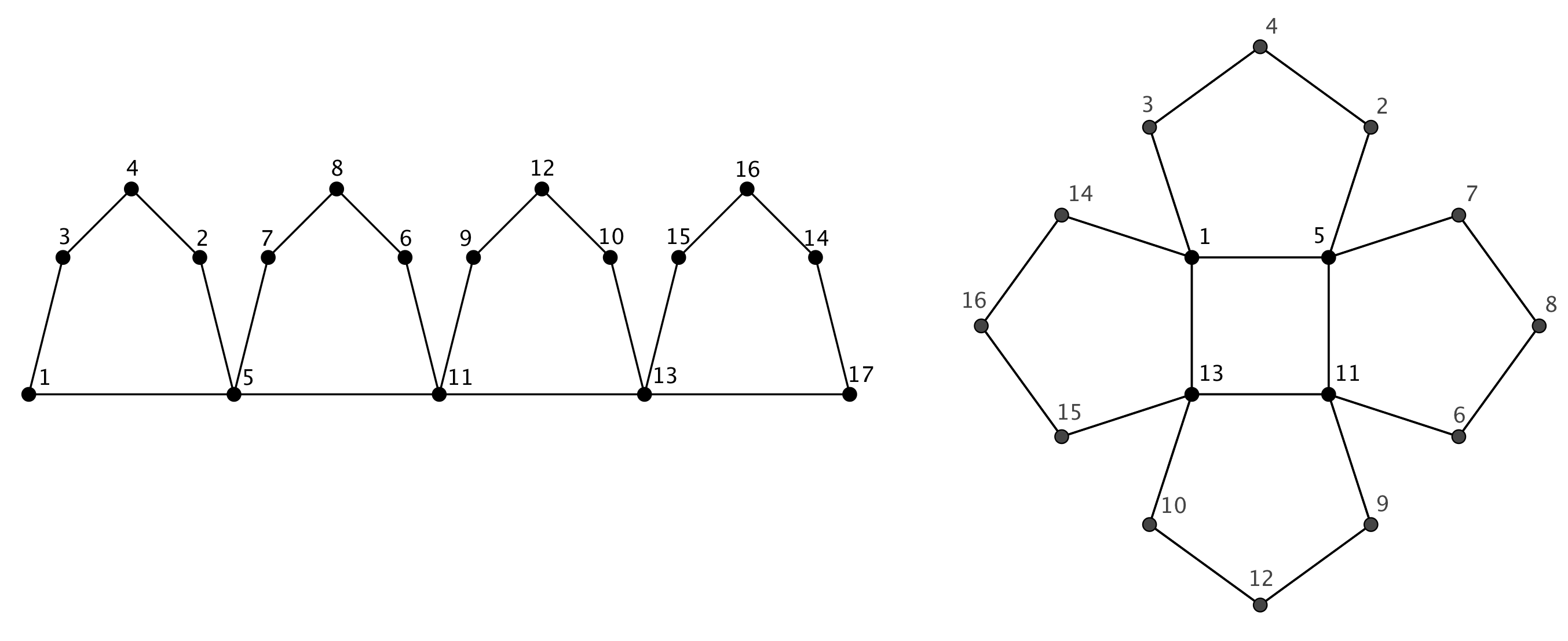}
\caption{Neighborhood-prime labeling of the pentagonal snake with $n=5$ and the star $(5,4)$-gon}\label{starmngon}
\end{center}
\end{figure}

Before investigating more general snake graphs, we first consider the operation of contracting two vertices of a neighborhood-prime graph, particularly when this is applied to triangular, quadrilateral, and pentagonal snakes.  The \textit{contraction} of vertices $u_1$ and $u_2$ of a graph $G$ results in a graph $G'$ in which $u_1$ and $u_2$ are replaced by a vertex $w$ such that $N_{G'}(w)=N_G(u_1)\cup N_G(u_2)$.

\begin{theorem}\label{contraction}
Assume $G$ is a graph of order $n$ with neighborhood-prime labeling $f$, and let $u_1,u_2\in V(G)$ with $f(u_1)=1$ and $f(u_2)=n$.  If $u_1u_2$ is not an edge in $G$ and either $|N(u_1)|>1$ or $|N(u_2)|>1$, then the graph $G'$ resulting from contracting the vertices $u_1$ and $u_2$ into a vertex $w$ has a neighborhood-prime labeling $g$ in which $g(w)=1$ and $g(v)=f(v)$ for all other $v\in V(G')$.
\end{theorem}

\begin{proof}
	Let $v\in V(G')$ with $\deg(v)>1$. We will consider three cases.
	
\noindent\textbf{Case 1:}  Assume $u_{1}$ or $u_2\in N_G(v)$.  Note that by our assumption of $u_1$ and $u_2$ not being adjacent in $G$, we have $v\neq w$.  Then $w\in N_{G'}(v)$, and since $g(w)=1$, $\gcd\{g(N_{G'}(v))\}=1$.
	
\medskip	
\noindent\textbf{Case 2:} Next assume $v\neq w$ and $u_1,u_2\notin N_G(v)$.  Then we have $N_{G'}(v)=N_G(v)$.  Thus, $\gcd\{g(N_{G'}(v))\}=\gcd\{f(N_G(v))\}=1$ since $f$ is a neighborhood-prime labeling of $G$.
	
\medskip
\noindent\textbf{Case 3:} For our last case, let $v=w$ and without loss of generality, assume $|N(u_1)|>1$.  Then $N_G(u_1)\subseteq N_{G'}(v)$. Since $\gcd\{f(N_{G}(u_1))\}=1$, we can conclude that $\gcd\{g(N_{G'}(v))\}=1$ based on the subset relationship between the neighborhoods. 
\medskip

These cases cover all possible vertices $v$; thus, the labeling $g$ is neighborhood-prime.

\end{proof}

The \textit{star (k,n)-gon} is a graph constructed by adjoining the endpoints of a path $P_{k-2}$ to the endpoints of each edge of a cycle $C_n$, as shown in Figure~\ref{starmngon} for the case of $k=5$ and and $n=4$.  This graph can also be viewed as a snake graph of length $n$ in which the head and tail of the snake are contracted.  It should be noted that Seoud and Youssef~\cite{S_Y} proved that the star $(k,n)$-gon is prime for all $k,n\geq 3$, which implies that all snakes are prime through the removal of the last length $k$ cycle.  We use the opposite approach through the vertex contraction perspective to prove star $(k,n)$-gons are neighborhood-prime for certain $k$ values based on our snake results.

\begin{corollary}
The star $(k,n)$-gon is neighborhood-prime for all $n\geq 3$ when $k=3$, $4$, or $5$.
\end{corollary}
\begin{proof}
This is a direct application of Theorem~\ref{contraction} by contracting the vertices $u_1$ and $u_{n+1}$ at the ends of the path of a snake graph $S_{k,n+1}$.  The required assumptions are met since the labelings in Theorems~\ref{triangular_snake}, \ref{quad_snake}, and \ref{pentagon_snake} each have $f(u_1)=1$ and $f(u_{n+1})=|V(G)|$ with these vertices not being adjacent and with each having degree 2.
\end{proof}

Our previous results on triangular, quadrilateral, and pentagonal snakes can be greatly expanded upon for more general cases of the type of polygon.  Unfortunately, Theorem~\ref{contraction} cannot be applied to the upcoming graphs, because the labeling on the head and tail of the snake differs from the previous snake graphs.  We consider the general $k$-\textit{polygonal snake}, denoted by $S_{k,n}$, to be a path on $n$ vertices with endpoints of paths $P_k$ being attached to each pair of adjacent vertices.  To find a neighborhood-prime labeling for more certain larger values of $k$, we change our perspective of the snake graph by considering the vertices $v_1,\ldots, v_{|S_{k,n}|}$ as belonging to a single path of length $m=(n-1)(k-1)+1$ with additional edges from $v_{i(k-1)+1}$ to $v_{(i+1)(k-1)+1}$ for $i=0,\ldots, n-2$ that form the base of the snake.  

Using this perspective, we can label the $m$ vertices, for certain values of $k$ and $n$, using Patel and Shrimali's function for a neighborhood-prime labeling of a path, given in~\cite{Patel_Shrimali1}.  We have rewritten its formula slightly as the following for a path $P_n$ with vertices $\{v_1,v_2,\ldots, v_n\}$, where $f: V(P_n)\rightarrow \{1,2,\ldots, n\}$
\begin{equation}\label{path}
f(v_{i})=\begin{cases}
\text{  }\lfloor\frac{n}{2}\rfloor+\frac{i+1}{2} \hspace{.6cm}\text{if $i$ is odd}\\
\text{  }\frac{i}{2} \hspace{2cm}\text{if $i$ is even.}\\
\end{cases}
\end{equation}
This function is used to create a neighborhood-prime labeling for the snake $S_{9,3}$ in Figure~\ref{kgonsnake}.  To demonstrate that our labeling for snake graphs is neighborhood-prime, we will use the following result involving greatest common divisors.

\begin{remark}
For all integers $a$, $b$, $c$, and $d$, the following are true.
\begin{align}
\gcd\{a,b\}&=\gcd\{c\cdot a+d\cdot b,b\}  \label{gcd1} \\
\gcd\{a,b\}&=\gcd\{a,c\cdot a+d\cdot b\} \label{gcd2}.
\end{align}
\end{remark}

Note that the inequalities in the next theorem for the values of $k$ and $n$ are set to not repeat previous results and to avoid $n=2$, which is the case of a single cycle.

\begin{theorem}\label{kpolygonalsnake}
The $k$-polygonal snake $S_{k,n}$ has a neighborhood-prime labeling for the following cases where $k\geq 6$ and $n\geq 3$:
\begin{itemize}
\item $k\equiv 1\pmod{4}$ {\rm and }$n=2^{\ell}+1$ {\rm for }$\ell\geq 1$

\item $k\equiv 0\pmod{4}$ {\rm and }$n=2^{\ell}$ {\rm for }$\ell\geq 2$

\item $k\equiv 0\pmod{4}$ {\rm and }$n=2^{\ell}+1$ {\rm for} $\ell\geq 1$

\item $k=2^{\ell}+2$  {\rm for }$\ell\geq 2$ {\rm and }$n\equiv 3\pmod{4}$

\item $k$ {\rm even and }$n=3$

\item $k=2^{\ell}+3$  {\rm for }$\ell\geq 2$ {\rm and }$n$ {\rm even.}

\end{itemize}
\end{theorem}

\begin{figure}
\begin{center}
\includegraphics[scale=1.1]{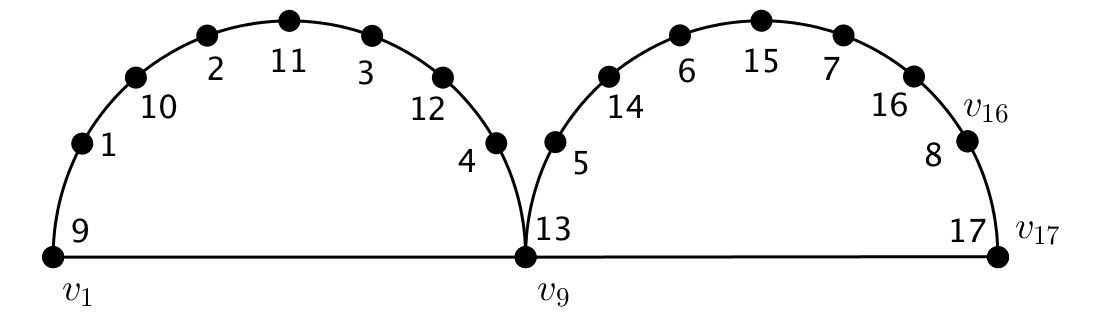}
\caption{Neighborhood-prime labeling of the snake $S_{9,3}$}\label{kgonsnake}
\end{center}
\end{figure}

\begin{proof}
We label the graph $S_{k,n}$ by assigning to the vertices $v_1,\ldots v_m$ with $m=(n-1)(k-1)+1$ the values from the function described in Equation~\eqref{path}, where we consider the path $P_m$ as tracing the edges of each cycle that are not along the base path as shown in Figure~\ref{kgonsnake}.  Note that the labeling function depends on the parity of $m$, which is determined by the cases of $k$ and $n$ being even or odd.

Let $v\in V(S_{k,n})$.  For $v=v_1$, since $v_2\in N(v_1)$ and $f(v_2)=1$ for any case of $k$ and $n$, we have $\gcd\{f(N(v_1))\}=1$.

When $v=v_i$ for $i=2,\ldots, m-1$, we have $\gcd\{f(N_{S_{k,n}}(v))\}=\gcd\{f(N_{P_m}(v))\}=1$ 
since the labeling on the subgraph $P_m$ is neighborhood-prime and $N_{P_m}(v)\subseteq N_{S_{k,n}}(v)$.

For the vertex $v=v_m$, we observe the neighborhood of this vertex to be $N(v_m)=\{v_{m-k-1},v_{m-1}\}$.  The labels for these vertices depend on the parity of $k$ and $n$.  
For the case of $k\equiv 1\pmod{4}$ and $n=2^{\ell}+1$ for $\ell\geq 1$, an example of which can be seen in Figure~\ref{kgonsnake}, $k$ and $n$ are both odd.  This implies $m$ is odd, $m-1$ is even, and $m-k+1$ is odd, where we recall $m=(n-1)(k-1)+1$.  Then Equation~\eqref{path} provides that $f(v_{m-k+1})=\frac{m-1}{2}+\frac{m-k+2}{2}$ and $f(v_{m-1})=\frac{m-1}{2}$.  Therefore, we have
\begin{align*}
\gcd\{f(N(v_m))\}&=\gcd\left\{\frac{m-1}{2}+\frac{m-k+2}{2},\frac{m-1}{2}\right\}  \\
&=\gcd\left\{\frac{m-k+2}{2},\frac{m-1}{2}\right\}\text{(using Eq.~\eqref{gcd1} with $c=1$ and $d=-1$)} \\
&=\gcd\left\{\frac{k-3}{2},\frac{m-1}{2}\right\}  \text{(using Eq.~\eqref{gcd1} with $c=-1$ and $d=1$)} \\ 
&=\gcd\left\{\frac{k-3}{2},\frac{(n-1)(k-1)}{2}\right\}\\
&=\gcd\left\{\frac{k-3}{2},\frac{(n-1)k-(n-1)}{2}\right\}\\
&=\gcd\left\{\frac{k-3}{2},\frac{2n-2}{2}\right\} \text{(using Eq.~\eqref{gcd2} with $c=-(n-1)$ and $d=1$)} \\ 
&=\gcd\left\{\frac{k-3}{2},n-1\right\}.
\end{align*}

When $k\equiv 1\pmod{4}$ or $k=4a+1$ for some integer $a$, $\frac{k-3}{2}=2a-1$ is odd.  For $n=2^{\ell}+1$, the value $n-1$ is a power of $2$.  Thus, we have that $\gcd\left\{\frac{k-3}{2},n-1\right\}=1$, making the labeling neighborhood-prime in this case.

In the case of $k\equiv 0\pmod{4}$ {\rm and }$n=2^{\ell}$ {\rm for }$\ell\geq 2$, the neighborhood of $v_m$ remains the same.  However, with $k$, $n$, and $m$ being even, Equation~\eqref{path} results in $f(v_{m-k+1})=\frac{m}{2}+\frac{m-k+2}{2}=m-\frac{k-2}{2}$ and $f(v_{m-1})=\frac{m}{2}+\frac{m}{2}=m$ since both $m-k+1$ and $m-1$ are odd.  Thus, we have
\begin{align*}
\gcd\{f(N(v_m))\}&=\gcd\left\{m-\frac{k-2}{2},m\right\}\\
&=\gcd\left\{\frac{k-2}{2},m\right\}  \text{(using Eq.~\eqref{gcd1} with $c=-1$ and $d=1$)} \\ 
&=\gcd\left\{\frac{k-2}{2},(n-1)(k-1)+1\right\} \\
&=\gcd\left\{\frac{k-2}{2},(n-1)k-n+2\right\} \\
&=\gcd\left\{\frac{k-2}{2},n\right\} \text{(using Eq.~\eqref{gcd2} with $c=-2(n-1)$ and $d=1$)}
\end{align*}

When $k\equiv 0\pmod{4}$, the integer $\frac{k-2}{2}$ is odd, which is relatively prime with $n$ since it is assumed to be a power of $2$, making this $\gcd$ equal to $1$.

For each of the three cases of $k\equiv 0\pmod{4}$ and $n=2^{\ell}+1$ for $\ell\geq 1$, $k=2^{\ell}+2$ $\ell\geq 2$ and $n\equiv 3\pmod{4}$, and $k$ even and $n=3$, we have that $k$ is even while $n$ is odd.  Therefore, the indices $m-k+1$ and $m-1$ are both even and thus $f(v_{m-k+1})=\frac{m-k+1}{2}$ and $f(v_{m-1})=\frac{m-1}{2}$.  Following similar reasoning as in the previous cases, we can compute the following
$$\gcd\{f(N(v_m))\}=\gcd\left\{\frac{m-k+1}{2},\frac{m-1}{2}\right\}=\gcd\left\{\frac{k-2}{2},\frac{n-1}{2}\right\}.$$
When $k\equiv 0\pmod{4}$ and $n=2^{\ell}+1$, $\frac{k-2}{2}$ is odd, and $\frac{n-1}{2}$ is a power of $2$, making their $\gcd$ be $1$.  The case of $k=2^{\ell}+2$ and $n\equiv 3\pmod{4}$ follows similarly with $\frac{k-2}{2}$ being a power of $2$ and $\frac{n-1}{2}$ being odd.  If $n=3$, then the second value in the pair is $1$, so any $k$ that is even results in $\gcd\{f(N(v_m))\}=1$.

The final case of $k=2^{\ell}+3$ for $\ell\geq 2$ and $n$ even results in $m-1$ being even and $m-k+1$ being odd.  Therefore, the labels on the vertices $v_{m-k+1}$ and $v_{m-1}$ are the same as in the first case, which implies

$$\gcd\{f(N(v_m))\}=\gcd\left\{\frac{m-1}{2}+\frac{m-k+2}{2},\frac{m-1}{2}\right\}=\gcd\left\{\frac{k-3}{2},n-1\right\}.$$
With $k=2^{\ell}+3$, $\frac{k-3}{2}$ is a power of $2$, and $n-1$ is odd in this case.  Hence, $\gcd\left\{\frac{k-3}{2},n-1\right\}=1$, which concludes our last case for showing $\gcd\{f(N(v_m))\}=1$, proving that this labeling is neighborhood-prime.

\end{proof}

One interesting outcome from Theorem~\ref{kpolygonalsnake} regards the hexagonal snake, or when $k=6$.  The fourth case included in the theorem provides a neighborhood-prime labeling if $n\equiv 3\pmod{4}$.  This is of note since the hexagonal cycle is not neighborhood-prime, as shown by~\cite{Patel_Shrimali1}, but forming a snake graph with hexagons is neighborhood-prime for these lengths.

In addition to connecting polygons along a path to form a snake, another class of graph involving attaching polygons is the book graph.  The $k$-\textit{polygonal book}, denoted $B_{k,n}$, is formed by $n$ copies of a $k$-polygon sharing a single edge.  Each $k$-polygon is referred to as a page of the book graph.  A three page pentagonal book and its neighborhood-prime labeling can be seen in Figure~\ref{pentbook}.  The cases of $k=3$ and $k=4$, or the triangular and rectangular books, were shown to be neighborhood-prime by Ananthavalli and Nagarajan~\cite{A_N}.  We now extend their results to the pentagonal book.

\begin{figure}
\begin{center}
\includegraphics[scale=.2]{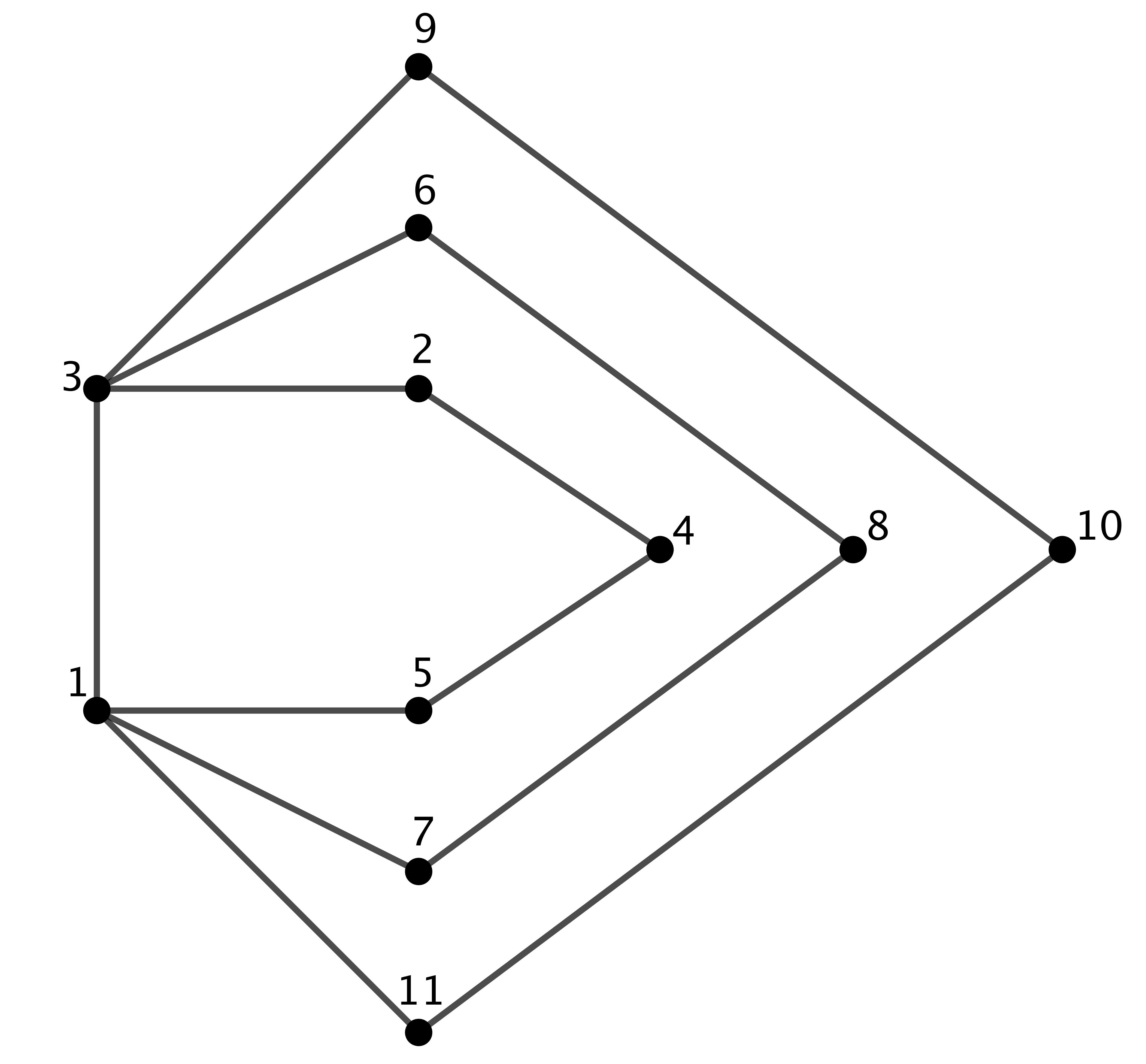}
\caption{Neighborhood-prime labeling of the pentagonal book $B_{5,3}$}\label{pentbook}
\end{center}
\end{figure}

\begin{theorem}
	The pentagonal book $B_{5,n}$ is neighborhood-prime for all $n$.
\end{theorem}
\begin{proof}
	Consider the first page to consist of the vertices $u_1,v_1,w_1,x_1,u_2$ with each subsequent page being a path $u_1,v_i,w_i,x_i,u_2$ for $i=2,\ldots, n$.  The vertices of the pentagonal book are labeled as follows.  The first page is labeled as $f(u_1)=3$, $f(v_1)=2$, $f(w_1)=4$, $f(x_1)=5$, and $f(u_2)=1$.  For the remaining pages in which $i=2,\ldots,n$, we assign
	\begin{align*}
	f(v_i)&=3i,\\
	f(w_i)&=\begin{cases}
		\text{  }3i+1 \text{ when }i \text{ is odd}\\
		\text{  }3i+2 \text{ when }i \text{ is even,} 
		\end{cases}\\		
	f(x_i)&=\begin{cases}
		\text{  }3i+2 \text{ when }i \text{ is odd}\\
		\text{  }3i+1 \text{ when }i \text{ is even.} 
		\end{cases}
		\end{align*}
Let $v\in V(B_{5,n})$.  For $v=v_1, w_1$, or $x_1$, one can see by inspection that its neighborhood has relatively prime labels.  When $v=u_1$ or $x_i$ for $i=2,\ldots, n$, we see that $u_2\in N(v)$.  Since $f(u_2)=1$, then $\gcd\{f(N(v))\}=1$.  If $v=u_2$, then $\{u_1,x_1\}\subseteq N(u_2)$.  Since the labels of these vertices are $3$ and $5$, $\gcd\{f(N(u_2))\}=1$.

The cases that remain are $v=v_i$ and $v=w_i$ for $i=2,\ldots, n$.  We have $N(v_i)=\{u_1,w_i\}$ where $f(u_1)=3$.  Since $f(w_i)=3i+1$ or $3i+2$, neither of which are multiples of $3$, we have $\gcd\{f(u_1),f(w_i)\}=1$.  Finally, we consider $v=w_i$, which has $\{v_i,x_i\}$ as its neighborhood.  For $i$ that are even, these two vertices have consecutive labels of $3i$ and $3i+1$.  If $i$ is odd, the labels $3i$ and $3i+2$ are consecutive odd integers.  In either case of the parity of $i$, $\gcd\{f(N(w_i))\}=1$.  This concludes the final case and proves the labeling is in fact neighborhood-prime.
\end{proof}

Many graphs, such as the complete graph, have been shown to not be prime, but they are in fact neighborhood-prime, or in the case of the wheel graph, it is not prime for certain cycle lengths, but neighborhood-prime for all lengths.  We conclude this section by considering the M{\"o}bius ladder and show that it falls into this latter category.  The \textit{M{\"o}bius ladder} $M_{2n}$ is formed by first considering the ladder graph $L_n=P_n\times P_2$, which is the Cartesian product on these two paths.  If we call the vertices along the two length $n$ paths $v_1,\ldots, v_n$ and $u_1,\ldots, u_n$, then the M{\"o}bius ladder results from the inclusion of edges $v_1u_n$ and $u_1v_n$, as shown in Figure~\ref{MobiusLadderGraph} for $n=6$.  Though it appears to still be an open problem to prove the M{\"o}bius ladder is prime when $n$ is odd, it has been shown to be not prime when $n$ is even based on a result by Fu and Huang~\cite{F_H} on the independence number of a prime graph.  We show that this graph is neighborhood-prime without regard to the parity of its cycle length.

\begin{theorem}
M{\"o}bius ladders $M_{2n}$ have a neighborhood-prime labeling for all $n$.
\end{theorem}

\begin{figure}
\begin{center}
\includegraphics[scale=.75]{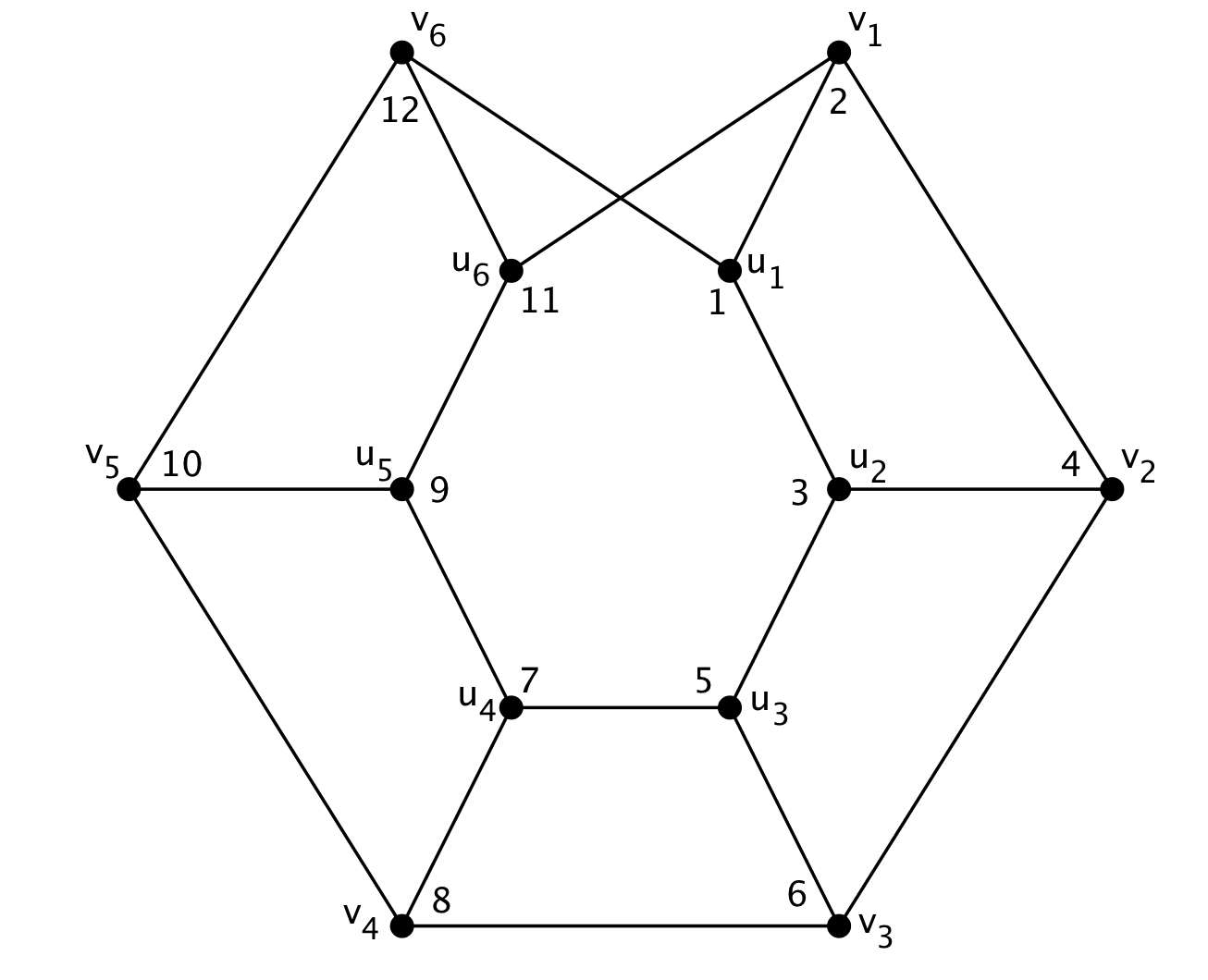}
\caption{Neighborhood-prime labeling of the M{\"o}bius ladder $M_{12}$}\label{MobiusLadderGraph}
\end{center}
\end{figure}

\begin{proof}
We label the vertices of $M_{2n}$ as follows: $f(u_i)=2i-1$ and $f(v_i)=2i$ for $i=1,\ldots, n$.  For $v\in V(M_{2n})$, the neighborhood of $v$ contains consecutive integer labels if $v=u_i$ for $i=1,\ldots, n-1$ or if $v=v_i$ for $i=2,\ldots, n$.  In the case of $v=v_1$, we have $u_1\in N(v_1)$, hence $\gcd\{f(N(v_1))\}=1$ because $f(u_1)=1$.  For the final case of $v=u_n$, $N(u_n)=\{u_{n-1}, v_n, v_1\}$.  Since $f(v_1)=2$ and $f(u_{n-1})$ is odd, the $\gcd$ of this set of labels is $1$, proving the labeling is neighborhood-prime.
\end{proof}

\section{Neighborhood-Prime Labelings of Trees}

Most neighborhood-prime labelings that have been found by previous authors have focused on graphs that contain cycles.  Some trees, however, have also been shown to be neighborhood-prime, such as paths, fans, and the sum $P_m+\overline{K_n}$ in~\cite{Patel_Shrimali1} and the bistar in~\cite{A_M}.

The labeling of a path is of particular interest, as paths are often building blocks for more complicated trees.  Patel and Shrimali provided a neighborhood-prime labeling of a path on $n$ vertices, as described in Equation~\eqref{path}.  This labeling will be used in several of our upcoming results on trees.  First, we demonstrate a way to add a pendant edge, or an edge connecting to a leaf, to a graph and maintain a neighborhood-prime labeling.

\begin{theorem}\label{pendent_vertex}
Assume $G$ is a neighborhood-prime graph with $|V(G)|=n$.  Let $v$ be a vertex in $G$ with $\deg(v)>1$.  The graph $G'$ obtained by attaching a pendent edge to the vertex $v$ is neighborhood-prime.
\end{theorem}

\begin{proof}
Let $f: V(G)\rightarrow \{1,2,\ldots,n\}$ be the neighborhood-prime labeling for $G$, and let $v'$ be the new vertex adjacent to $v$, resulting in $V(G')=V(G)\cup \{v'\}$.  We define a function $g: V(G')\rightarrow \{1,2,\ldots,n,n+1\}$ by assigning $g(v')=n+1$ and $g(u)=f(u)$ for all $u\in V(G)$.  

Let $u$ be a vertex in $G'$ with $\deg(u)>1$, which note is not $v'$ since it is a pendant vertex.
For the case of $u=v$, we have $N_{G'}(v)= N_{G}(v)\cup \{v'\}$.  Therefore, $\gcd\{g(N_{G'}(v))\}=\gcd\{f(N_G(v))\}=1$ since $N_G(v)\subset N_{G'}(v)$ and because $f$ is a neighborhood-prime labeling of $G$.

In the case of $u\neq v$, $N_{G'}(u)= N_{G}(u)$, so $\gcd\{g(N_{G'}(u))\}=1$.  Thus, we have that $g$ is a neighborhood-prime labeling of $G'$.
\end{proof}

The first class of tree that we consider is the \textit{caterpillar}, which is a tree with all vertices being within distance $1$ of a central path.  Caterpillars can be constructed from a path $P_n$ in which each of the $n-2$ interior vertices either remains as degree $2$ or is adjacent to at least one pendent vertex, as seen in Figure~\ref{caterpillar}.  This structure allows us to use Theorem~\ref{pendent_vertex} to find a neighborhood-prime labeling.  Note that our following result on labeling the caterpillar generalizes the bistar graph that was examined in~\cite{A_M}.

\begin{figure}
\begin{center}
\includegraphics[scale=.95]{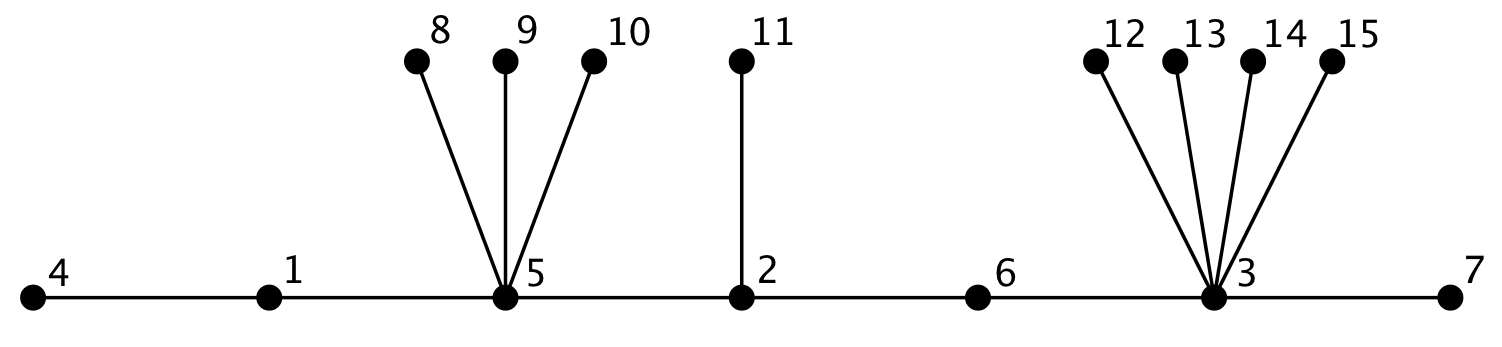}
\caption{Neighborhood-prime labeling of a caterpillar}\label{caterpillar}
\end{center}
\end{figure}

\begin{corollary}\label{caterpillars}
All caterpillars have neighborhood-prime labelings.
\end{corollary}

\begin{proof}
Consider a caterpillar with a central path $P_n$.  First label the path using Patel and Shrimali's path labeling given in Equation~\eqref{path}.  Then since the interior path vertices were initially degree $2$, Theorem~\ref{pendent_vertex} provides a way to assign labels to each pendent vertex while maintaining the neighborhood-prime labeling condition at each step.
\end{proof}


	

	

A \textit{spider} is another class of trees that can be shown to neighborhood-prime. It is defined as a tree with only one vertex of degree $3$ or more, but a spider can also be viewed as a collection of paths $P_{n_1}$, \ldots, $P_{n_k}$ with one end of each path adjoined to a central vertex.  Figure~\ref{spider} shows a spider with a neighborhood-prime labeling.  Spiders have also been proven to have prime labelings, as seen in~\cite{Lee} by Lee, Wui, and Yeh.  

\begin{theorem}
All spiders are neighborhood-prime.
\end{theorem}

\begin{proof}
We begin the labeling by assigning the label $1$ to the central vertex.  For each path, we apply a similar function to the path labeling from Equation~\eqref{path}.  We assign the smallest label at the first vertex in each path though instead of the second vertex, and we must shift the labels to begin each path with the smallest available label.  Consider the labeling function for a set of vertices $v_1,\ldots, v_m$ with $N\in \mathbb{N}$:
$$f_N(v_i)=\begin{cases}
\text{  }N+\frac{i+1}{2} \hspace{1.8cm}\text{if $i$ is odd}\\
\text{  }N+\lfloor\frac{m}{2}\rfloor+\frac{i}{2} \hspace{1cm}\text{if $i$ is even.}\\
\end{cases}$$
For each path $P_{n_j}$, if we consider the vertices as $v_1,\ldots, v_{n_j}$ with $v_1$ being adjacent to the center and $v_{n_j}$ as the leaf, we assign labels using $f_N(v_i)$ in which $m=n_j$ and $N=1+n_1+\cdots +n_{j-1}$.  In the case of every length $n_j$ being even, we reflect the labeling of the last path by swapping the labels of $f_N(v_{n_j-i+1})$ and $f_N(v_{i})$ for each $i=1,\ldots, n_j$.  This reassignment can be seen in the labeling on Figure~\ref{spider} on the path $P_6$.

Consider a vertex $v$ with degree at least $2$.  If $v$ is on a path $P_{n_j}$ and is adjacent to the central vertex, then $f(N(v))$ contains the label $1$.  If $v$ is on a path, but not adjacent to the central vertex, its neighborhood is labeled by consecutive integers.  Finally, we consider when $v$ is the central vertex.  The label of the first vertex of $P_{n_1}$ is $2$, and if any of the lengths of the paths are odd, then the first vertex of the path immediately following the first odd $n_{i}$ will be labeled with an odd integer.  If none of the lengths are odd, then the reflection that is made on $P_{n_j}$ will result in the first vertex of that path having an odd label.  Since $\gcd\{2,\ell\}=1$ for any odd $\ell$, the neighborhood of the central vertex will have a relatively prime set of labels, proving the labeling is neighborhood-prime.
\end{proof}

\begin{figure}
\begin{center}
\includegraphics[scale=.75]{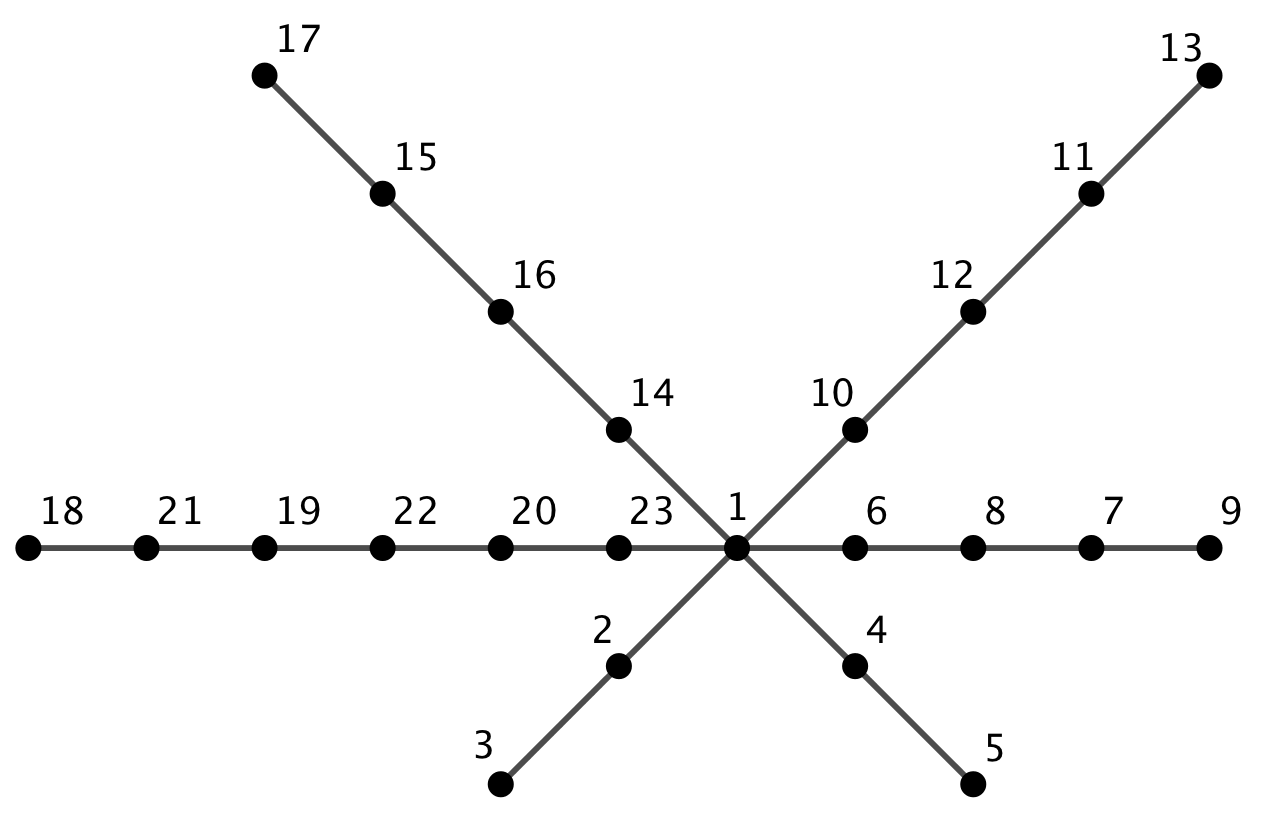}
\caption{Neighborhood-prime labeling a spider with paths $P_2$, $P_2$, $P_4$, $P_4$, $P_4$, and $P_6$}\label{spider}
\end{center}
\end{figure}

The next tree we will examine is the $(n,k)$-\textit{banana tree}.  This graph is obtained by attaching a leaf from $n$ copies of a $k$-star to a single root vertex.  It has $n k+1$ vertices, and an example of this type of tree with a neighborhood-prime labeling can be seen in Figure~\ref{banana}.  We refer to the root vertex as $v'$, the leaf of the $i$th star that is adjacent to the root as $u_i$, and the center of each star as $w_i$.  We specifically investigate banana trees for $n\geq 3$ and $k\geq 4$ to avoid repetition of graphs that are caterpillars or spiders.

\begin{theorem}
The $(n,k)$-banana tree admits a neighborhood-prime labeling for all $n\geq 3$ and $k\geq 4$.
\end{theorem}

\begin{proof}
We label an $(n,k)$-banana tree $G$ with the following function $f: V(G)\rightarrow \{1,\ldots, n k+1\}$.  We first assign $f(v')=1$, $f(u_i)=i+1$, and $f(w_i)=(i-1)(k-1)+n+2$.  The remaining leaves of the $i$th star are labeled with the integers in $\{(i-1)(k-1)+n+3,\ldots,i(k-1)+n+1\}$, as seen in Figure~\ref{banana} for the case of $n=3$ and $k=6$.

To show this labeling is neighborhood-prime, we consider a vertex $v\in V(G)$ in which $\deg(v)>1$.  The first case is $v=v'$, whose neighborhood is $\{u_1,\ldots, u_n\}$.  This neighborhood contains consecutive integers, and hence $\gcd\{f(N(v'))\}=1$.

The second case is $v=w_i$ for some $i=1,\ldots, n$.  Since $k\geq 4$, $N(w_i)$ contains two leaves labeled by consecutive integers.  Finally, if $v=u_i$, we have that $v'\in N(u_i)$.  Since $f(v')=1$, the $\gcd$ of the labels in its neighborhood is also $1$.  Thus, the labeling is neighborhood-prime.
\end{proof}

\begin{figure}
\begin{center}
\includegraphics[scale=.75]{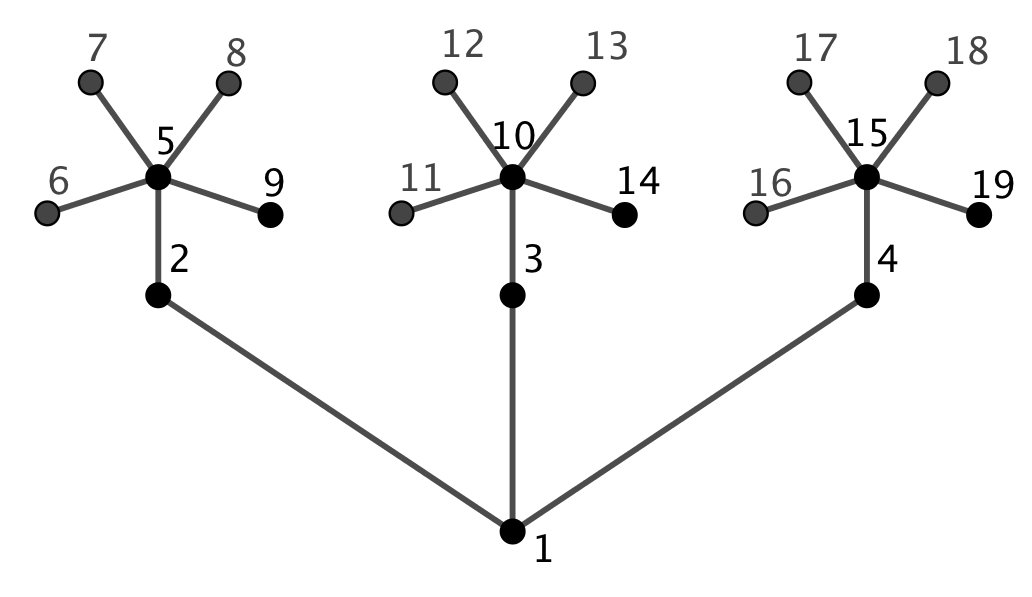}
\caption{Neighborhood-prime labeling of the banana tree $B_{3,6}$}\label{banana}
\end{center}
\end{figure}

An $(n,k)$-\textit{firecracker}, denoted $F_{n,k}$ and also referred to as a \textit{palm tree}, is a path $P_n$ in which each vertex is also a leaf on a $k$-star.  See Figure~\ref{firecracker} for examples of  $F_{6,3}$ and $F_{3,5}$ with neighborhood-prime labelings.  We will only consider $k\geq 3$ since $k=1$ is a path and $k=2$ is a caterpillar, both of which are already proven to be neighborhood-prime.  For the case of $k=3$, we will call the vertices along the path $u_1,\ldots, u_n$, the vertices adjacent the path $v_1,\ldots,v_n$, and the leaves $w_1,\ldots, w_n$ where $N(v_i)=\{u_i,w_i\}$.

Our labeling of firecrackers will rely on a coprime matching theorem by Pomerance and Selfridge~\cite{Pomerance} that was used in~\cite{Tout} to demonstrate a prime labeling of a cycle $C_n$ with~$t$ pendent vertices attached to each vertex.  A special case of their theorem that we will use relates to the bipartite graph with vertex set $X\cup Y$, in which $X=\{1,2,\ldots, n\}$ and $Y=\{2n+1,2n+2,\ldots, 3n\}$ for any integer $n$, where there exists an edge between $x\in X$ and $y\in Y$ if and only if $\gcd\{x,y\}=1$.  A perfect matching was proven to exist on this graph, providing a bijective function $g: X\rightarrow Y$ with $\gcd\{i,g(i)\}=1$ for all $i=1\,\ldots, n$.  The existence of this matching was also utilized by Robertson and Small~\cite{Robertson_Small} to demonstrate that firecrackers have a prime labeling.

\begin{theorem}\label{firecrackers}
The firecracker $F_{n,k}$ has a neighborhood-prime labeling for all $n\geq 1$ and $k\geq 3$.
\end{theorem}

\begin{proof}
We first consider the case of $k=3$ by assigning labels using a function $f: V(F_{n,3})\rightarrow \{1,2,\ldots, 3n\}$.  For the vertices along the path $P_n$, we label $u_1,\ldots,u_n$ using the Patel-Shrimali path labeling from Equation~\eqref{path}.  We then choose the label $f(v_n)=p$ where $p$ is a prime with $n+1\leq p\leq 2n$, which is guaranteed to exist for any $n$ by Bertrand's Postulate, a well-known number theory result related to the Prime Number Theorem.  For the vertices $v_1,\ldots,v_{n-1}$, we assign the labels $\{n+1,n+2,\ldots,2n\}\setminus \{p\}$ in any order.  Finally, for each vertex $w_i$, we set $f(w_i)=g(f(u_i))$ using the coprime mapping function $g$ from Pomerance and Selfridge's theorem.  

Let $v$ be an arbitrary vertex in $F_{n,3}$ with degree greater than $1$.  We consider four cases for this vertex $v$ to demonstrate this labeling is neighborhood-prime.
\medskip

\noindent\textbf{Case 1:} If $v=u_1$, then we see that $N(v)=\{v_1,u_2\}$.  We have $f(u_2)=1$ since the second vertex of any path is labeled by 1 according to Equation~\eqref{path}, so the $\gcd$ of this vertex's neighboring labels is $1$.
\medskip

\noindent\textbf{Case 2:} When $v$ is an interior path vertex, or $v=u_i$ for $i=2,\ldots, n-1$, then 
$\gcd\{f(N_{F_{n,3}}(v))\}=\gcd\{f(N_{P_n}(v))\}=1$ since $N_{P_n}(v)\subset N_{F_{n,3}}(v)$, and the labeling on the path $P_n$ is a neighborhood-prime labeling on that subgraph.

\medskip

\noindent\textbf{Case 3:} If $v=u_n$, then $N(v)=\{u_{n-1},v_n\}$.  Since $f(v_n)$ was chosen to be a prime $p$ satisfying $p>n\geq f(u_{n-1})$, the $\gcd$ of these two labels is $1$.

\medskip

\noindent\textbf{Case 4:} For the vertices $v=v_i$ for some $i=1,\ldots, n$, we observe $N(v)=\{u_i,w_i\}$.  By the coprime mapping theorem in~\cite{Pomerance}, $\gcd\{f(u_i),f(w_i)\}=\gcd\{f(u_i),g(f(u_i))\}=1$.

\medskip

We have shown that $F_{n,3}$ has a neighborhood-prime labeling.  By observing that $F_{n,k}$ for $k>3$ can be created from $F_{n,3}$ by attaching $k-3$ pendent vertices to each vertex $v_1,\ldots, v_n$, we have that $F_{n,k}$ is neighborhood-prime for all $k\geq 3$ by Theorem~\ref{pendent_vertex}.

\end{proof}
\begin{figure}
\begin{center}
\includegraphics[scale=.75]{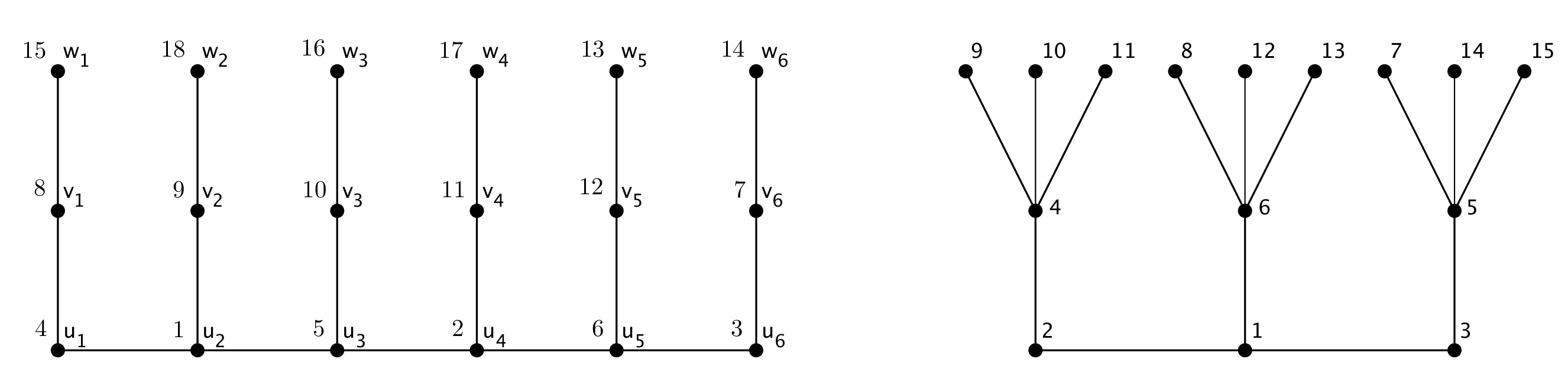}
\caption{Neighborhood-prime labeling of the firecrackers $F_{6,3}$ and $F_{3,5}$}\label{firecracker}
\end{center}
\end{figure}

We now consider a large set of trees that contains multiple classes of well-studied trees. 

\begin{definition}  
{\rm We call a graph $G$ a }\textit{bivalent-free tree} {\rm if $G$ is a tree in which each non-leaf vertex $v$ has $\deg(v)\geq 3$; i.e., no vertices are degree $2$}.
\end{definition}

\begin{theorem}\label{bivalent_free}
A graph $G$ if neighborhood-prime if either of the following conditions are met:
\begin{itemize}
\item $G$ is a bivalent-free tree

\item $G$ is a tree in which all of its degree $2$ vertices can be connected by a path between leaves.

\end{itemize}
\end{theorem}

\begin{proof}
First assume $G$ be a bivalent-free tree of order $n$, and consider a non-leaf vertex $v_1$ of $G$.  There exists a path $P_1$ from a leaf $u_1$ to another leaf $w_1$ that includes $v_1$ as an interior vertex.  We label this path $P_1$ using the Patel-Shrimali labeling defined in Equation~\eqref{path}.  If all vertices not on the path are leaves, then $G$ is a caterpillar that can be labeled according to Theorem~\ref{caterpillars}.  Otherwise, since each non-leaf vertex has at least $3$ neighbors, there is an interior vertex in $P_1$ with a non-leaf neighbor that lies outside the path, which we will call $v_2$.  With this vertex having at least two neighbors not lying on the path $P_1$ since it can only be adjacent to one vertex on the path to avoid a cycle, we can find another path $P_2$ connecting leaves $u_2$ and $w_2$ that passes through the vertex $v_2$.  We label this path using the following shifted path-labeling function, where we set $N=|P_1|$ and $m=|P_2|$ and consider the vertices $u_2,\ldots,v_2,\ldots, w_2$ of the path $P_2$ as $x_1,\ldots, x_{|P_{2}|}$:
\begin{equation}\label{shifted_path}
f_N(x_{i})=\begin{cases}
\text{  }N+\lfloor\frac{m}{2}\rfloor+\frac{i+1}{2} \hspace{.6cm}\text{if $i$ is odd}\\
\text{  }N+\frac{i}{2} \hspace{2cm}\text{if $i$ is even}\\
\end{cases}
\end{equation}
Note that this shifted labeling function still results in an interior vertex's neighbors within the path being assigned consecutive integers as their labels.

At this point, there may be several non-leaf vertices not contained in paths $P_1$ and $P_2$ that are adjacent to an interior vertex in one of those paths.  We add these vertices to the sequence of vertices $\{v_i\}$ in no particular order, and then continue this process of creating a path $P_i$ through $v_i$ from a leaf $u_i$ to another leaf $w_i$.  These paths exist since the degree of each $v_i$ is at least $3$ and because we claim only one edge from $v_i$ has been used to connect to a vertex on a path $P_k$ for some $k<i$.  If, on the other hand, a vertex $v_i$ was adjacent to vertices $y\in P_{k_1}$ and $z\in P_{k_2}$, then there would be a path from $y$ to $z$ since the subgraph induced by all vertices in $P_1,\ldots, P_{i-1}$ is connected.  The assumption that $v_i$ is adjacent to both of these vertices would result in a cycle, contradicting the fact that $G$ is a tree, and demonstrating our claim to be true.

For each path $P_i$ on the vertices $u_i,\ldots, v_i,\ldots, w_i$, we label the vertices using the shifted path-labeling function in Equation~\eqref{shifted_path} with $N=|P_1|+|P_2|+\cdots+|P_{i-1}|$ and $m=|P_{i}|$.  We continue creating these paths, adding additional vertices to the sequence $\{v_i\}$ which are non-leaves that are adjacent to any interior vertex of $P_i$, until the only vertices adjacent to the paths $P_1,\ldots, P_{\ell}$ are leaves.  Finally, we label these leaves that do not lie on any of our paths $P_i$ with the remaining labels between $|P_1|+\cdots +|P_{\ell}|+1$ and $n$ in any order.  See Figure~\ref{bivalentfree} for one possible result of this process of labeling a bivalent-free tree with the paths signified by the dashed edges.

\begin{figure}\label{bivalentfree}
\begin{center}
\includegraphics[scale=.95]{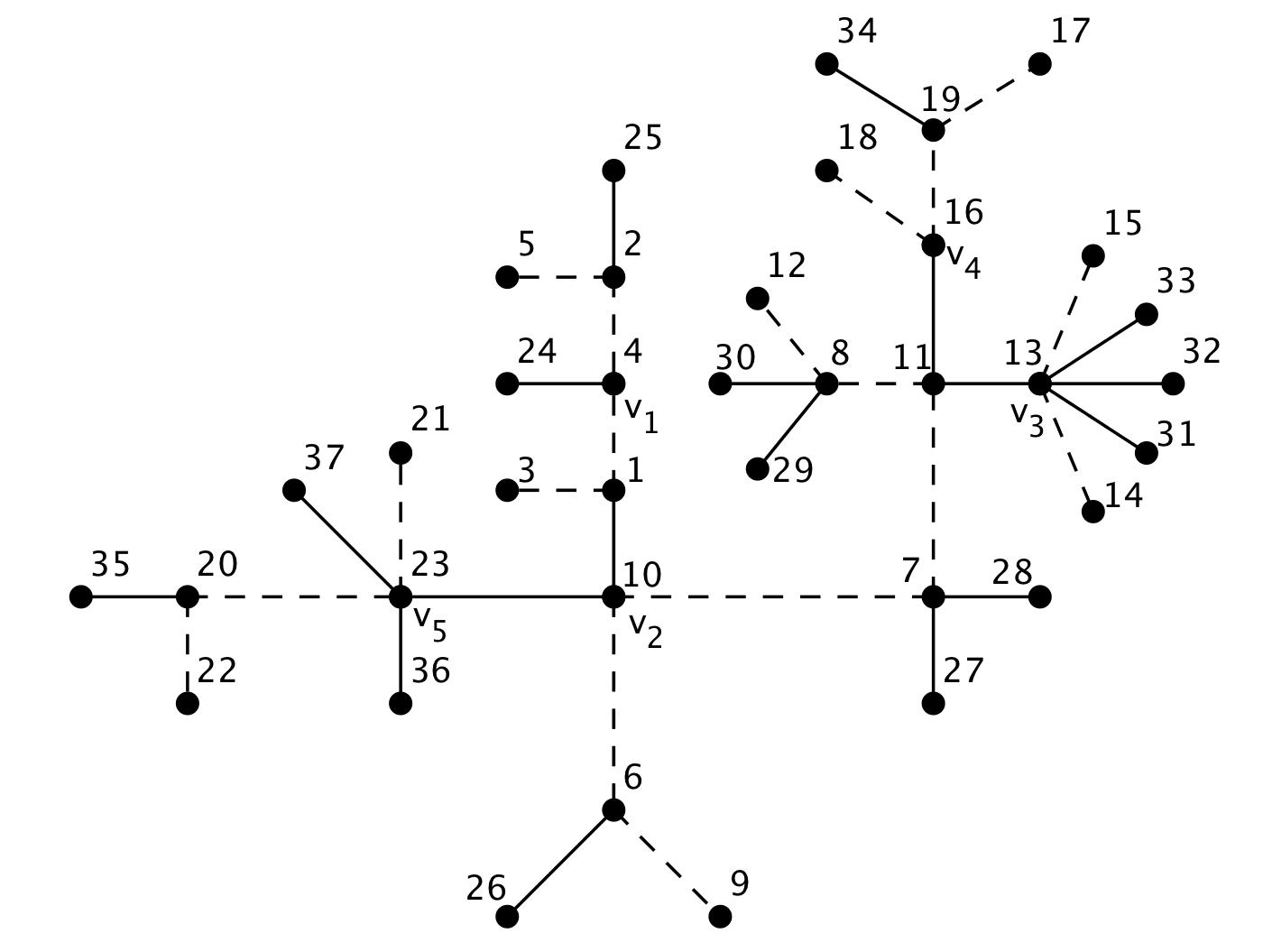}
\caption{Neighborhood-prime labeling of a bivalent-free tree}
\end{center}
\end{figure}

To show that the labeling is neighborhood-prime, consider a vertex $v\in V(G)$ with $\deg(v)>1$.  This vertex $v$ is an interior vertex of some path $P_k$, since the endpoints of each path are leaves.  Then $N(v)$ contains two vertices from the path $P_k$ that are labeled by consecutive integers.  Thus, $\gcd\{f(N(v))\}=1$ for all non-leaf vertices $v$, proving our result for bivalent-free trees.

In the case of $G$ having one or more degree $2$ vertices that can be connected by a single path from a leaf $u_1$ to another leaf $v_1$, we consider this path to be the path $P_1$ as described in the bivalent-free case.  If $G$ is simply a path or a caterpillar, the neighborhood-prime labeling is known to exist.  If it is not, then there must be a non-leaf adjacent to an interior vertex on $P_1$ to select as $v_2$, and its degree is at least $3$ since all degree $2$ vertices were on the path.  We continue from this point as in the previous case since all non-leaf vertices in $V(G)\setminus V(P_1)$ are at least degree $3$, resulting in a neighborhood-prime labeling in our second case.
\end{proof}

One class of trees that qualifies as a bivalent-free tree is a \textit{full} $k$-\textit{ary tree} with $k\geq 3$, which is a rooted tree in which each node is either a leaf or has $k$ children.  Additionally, $k$-\textit{Cayley trees} are bivalent-free if $k\geq 3$, where this class consists of all trees in which the non-leaf vertices have exactly degree $k$.  The lack of any bivalent vertices in these graphs leads Theorem~\ref{bivalent_free} to directly prove the following.

\begin{corollary}
Full $k$-ary trees and $k$-Cayley trees are neighborhood-prime when $k\geq 3$.
\end{corollary}

In addition to proving the existence of a neighborhood-prime labeling of bivalent-free trees such as $k$-Cayley trees and $k$-ary trees, the second case of Theorem~\ref{bivalent_free} provides a neighborhood-prime labeling of further classes of graphs.  First, caterpillars fit the criteria for this case, although proving the existence of their neighborhood-prime labeling separately served as a base case for this theorem.  Firecrackers in which $k>3$ also contain a path from a leaf of the first star to a leaf on the last that would contain all degree $2$ vertices, although the $k=3$ case would not have such a path.  

A third class of graphs that have not been examined yet, but have a neighborhood-prime labeling using Theorem~\ref{bivalent_free} involve binary trees, which are rooted trees with each node having at most $2$ children.  More specifically, a \textit{full binary tree} is a binary tree with each node having either $0$ or $2$ children.  Furthermore, a \textit{complete binary tree}, as shown in Figure~\ref{comp_bin_tree}, has each level completely filled except possibly the bottom level, in which all nodes are as far left as possible.  Fu and Huang~\cite{F_H} proved that complete binary trees are prime, although their definition of complete is usually referred to as a perfect binary tree, a subclass of complete binary trees in which all levels are filled.  

\begin{figure}
\begin{center}
\includegraphics[scale=1]{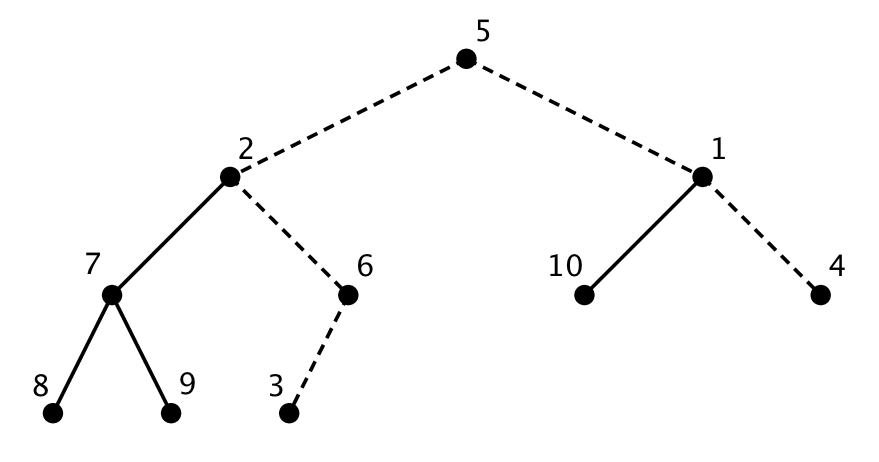}
\caption{Neighborhood-prime labeling of a complete binary tree}\label{comp_bin_tree}
\end{center}
\end{figure}

\begin{corollary}
Full binary trees and complete binary trees have neighborhood-prime labelings.
\end{corollary}
\begin{proof}
We apply the second case of Theorem~\ref{bivalent_free} since full binary trees only have the root as a degree $2$ vertex, which can easily be contained in a path between leaves.  For complete binary trees, the root and potentially one vertex in the penultimate row, again seen in Figure~\ref{comp_bin_tree}, have degree $2$.  One can again find a path from a leaf that passes through the root and this other vertex, if it exists, to its only child, as shown by the dashed edges in the figure.
\end{proof}

In addition to the previous result proving the existence of a neighborhood-prime labeling for these particular classes of binary trees, one can find a more explicit labeling.  Particularly for the full binary tree, simply label the root with $1$ and proceed by labeling in order from left to right across each level.  Each non-leaf will have its two children being labeled by consecutive integers.

Based on the wide variety of classes of trees with neighborhood-prime labelings, we conclude with the following conjecture. 

\begin{conjecture}
All trees are neighborhood-prime.
\end{conjecture}

\newcommand{\journal}[6]{{#1,} {#2}, {\it #3} {\bf #4} (#5) #6.}
\newcommand{\dissertation}[4]{{#1,} #2, {\it #3,} #4}
\newcommand{\book}[5]{{#1,} {\it #2,} #3, #4.}

\end{document}